\documentclass[12pt]{article}
\usepackage{epsf}

\usepackage{amsthm}
\usepackage{amssymb}
\usepackage{amsmath}

\usepackage[top=1.5cm,bottom=1.5cm,left=2.5cm,right=2.5cm,includehead,includefoot
           ]{geometry}

\makeatletter

\numberwithin{equation}{section}
\newtheorem{theorem}{Theorem}[section]
\newtheorem{lemma}[theorem]{Lemma}
\newtheorem{corollary}[theorem]{Corollary}
\newtheorem{remark}[theorem]{Remark}

\newtheorem{example}[theorem]{Example}

\title{Characterizations of 2-local derivations and local Lie derivations on some algebras}
\author{\begin{tabular}{c} Jun He\footnote{Corresponding author.
E-mail address: 15121034934@163.com} , Jiankui Li, Guangyu An, and Wenbo Huang
\\{\small\it Department of Mathematics, East China University of
Science and Technology}\\
{\small\it Shanghai 200237, China}
\end{tabular}}
\date{}
\begin{document}
\maketitle \abstract We prove that every 2-local derivation from the
algebra $M_n(\mathcal{A})(n>2)$ into its bimodule $M_n(\mathcal{M})$ is a derivation, where $\mathcal{A}$ is a unital Banach algebra and $\mathcal{M}$ is a unital $\mathcal{A}$-bimodule such that each Jordan derivation from $\mathcal{A}$ into $\mathcal{M}$ is an inner derivation, and that every 2-local derivation on a C*-algebra with a faithful
traceable representation is a derivation. We also characterize local and 2-local Lie
derivations on some algebras such as von Neumann algebras, nest algebras, Jiang-Su algebra and UHF algebras.

\

\

\
\section{Introduction}\

Let $\mathcal{A}$ be a Banach algebra and $\mathcal{M}$  an
$\mathcal{A}$-bimodule. We recall that a linear map $D:$ $\mathcal{A}$
$\rightarrow$ $\mathcal{M}$ is called a derivation if
$D(ab)=D(a)b+aD(b)$ for all $a,b\in$ $\mathcal{A}$, and a linear map $D:$ $\mathcal{A}$
$\rightarrow$ $\mathcal{M}$ is called a Jordan derivation if
$D(a^2)=D(a)a+aD(a)$ for all $a\in$ $\mathcal{A}$.
A derivation $D_a$ defined by $D_{a}(x)=ax-xa$ for all $x\in \mathcal{A}$ is called an inner derivation, where $a$ is a fixed element in $\mathcal{M}$.

In \cite{R. Kadison}, R. Kadison introduces the concept of local derivation in the following sense:
a linear mapping $T$ from $\mathcal{A}$ into $\mathcal{M}$ such that
for every $a\in$ $\mathcal{A}$, there exists a derivation $D_a :$
$\mathcal{A}$ $\rightarrow$ $\mathcal{X}$, depending on $a$,
satisfying $T(a)=D_a(a)$. Also in \cite{R. Kadison}, the author proves that each continuous local
derivation from a von Neumann algebra into its dual Banach
module is a derivation. B. Jonson \cite{B. Johnson} extends the above result
by proving that every local derivation from a C*-algebra
into its Banach bimodule is a derivation. Based on these results,
many authors have studied local derivations on operator algebras, for
example, see in \cite{R. Crist, D. Hadwin, D. Larson, Y. Pang}.

In \cite{P. Semrl} P. $\check{S}$emrl introduces the concept of 2-local derivations. Recall
that a map $\Delta$ : $\mathcal{A}$ $\rightarrow$ $\mathcal{M}$ (not necessarily linear)
is called a 2-local derivation if for each $a,b\in$ $\mathcal{A}$, there exists a derivation
$D_{a,b}$ : $\mathcal{A}$ $\rightarrow$ $\mathcal{M}$ such that  $\Delta(a)=D_{a,b}(a)$ and
$\Delta(b)=D_{a,b}(b)$. Moreover, the author proves that every 2-local derivation
on $B(H)$ is a derivation. In \cite{S. Kim} S. Kim and J. Kim give a short proof of that every 2-local derivation
on the algebra $M_{n}(\mathbb{C})$ is a derivation. Later J. Zhang and H. Li \cite{J. Zhang} extend the above
result for arbitrary symmetric digraph matrix algebras and construct an example of 2-local derivation
which is not a derivation on the algebra of all upper triangular complex $2\times2$ matrices.

In \cite{S. Ayupov 1}, S. Ayupov and K. Kudaybergenov suggest a new technique and prove that every 2-local derivation
on $B(H)$ is a derivation for arbitrary Hilbert space $H$. Then they consider the cases for several kinds of
von Neumann algebras in succession in \cite{S. Ayupov 2,S. Ayupov 3,S. Ayupov 4}, and finally prove that any 2-local derivation on arbitrary von Neumann algebra
is a derivation. Quite recently, in \cite{S. Ayupov 5} the authors study the case for matrix algebras over unital semiprime Banach algebras,
and prove that every 2-local derivation on the algebra $M_{2^{n}}(\mathcal{A})$, $n\geq2$, is a derivation, where $\mathcal{A}$
is a unital semiprime Banach algebra with the inner derivation property.

In Section 2, we improve the above result(\cite[Theorem 2.1]{S. Ayupov 5}) for arbitrary $n>2$.
More specifically, we prove that if $\mathcal{A}$ is a unital Banach algebra and $\mathcal{M}$ is a unital $\mathcal{A}$-bimodule such that each Jordan derivation from $\mathcal{A}$ into $\mathcal{M}$ is an inner derivation, then each 2-local derivation $\Delta$ from $M_n(\mathcal{A})(n>2)$ into $M_n(\mathcal{M})$ is a derivation. Moreover, if we only consider the case $\mathcal{M}=\mathcal{A}$, then the assumption of innerness can be relaxed to spatial innerness. That is, if $\mathcal{A}$ is a unital Banach algebra such that each Jordan derivation on $\mathcal{A}$ is a derivation and for each derivation $D$ on $\mathcal{A}$, there exists an element $a$ in $\mathcal{B}$ such that $D(x)=[a,x]$ for all $x\in \mathcal{A}$, where $\mathcal{B}$ is an algebra containing $\mathcal{A}$, then each 2-local derivation on $M_n(\mathcal{A})(n>2)$ is a derivation. Based on these, we obtain some applications. We also prove
that every 2-local derivation on a C*-algebra with a faithful traceable representation is a derivation.

Recall that a linear map $\varphi:\mathcal{A}\rightarrow \mathcal{A}$ is called a Lie derivation if $\varphi[a,b]=[\varphi(a),b]+[a,\varphi(b)]$, for all $a,b\in \mathcal{A}$, where $[a,b]=ab-ba$ is the usual Lie product. A Lie derivation $\varphi$ is said to be standard if it can be decomposed as $\varphi=D+\tau$, where $D$ is a derivation on $\mathcal{A}$ and $\tau$ is a linear map from $\mathcal{A}$ into the center of $\mathcal{A}$ such that $\tau[a,b]=0$ for all $a,b\in \mathcal{A}$.

It is natural to ask under which conditions each Lie derivation is standard. This problem has been studied by many authors. M. Mathieu and A. Villena \cite{M. Mathieu} prove that each Lie derivation on a C*-algebra is standard. W. Cheung \cite{W. Cheung} characterizes Lie derivations on triangular algebras. F. Lu \cite{F. Lu3} proves that each Lie derivation on a CDCSL(completely distributed commutative subspace lattice) algebra is standard.

Obviously, one can define local and 2-local Lie derivations in a similar way as the local and 2-local derivations. In \cite{L. Chen}, L. Chen, F. Lu and T. Wang prove that every local Lie derivation on $B(X)$ is a Lie derivation, where $X$ is a Banach space of dimension exceeding 2. Also in this paper \cite{L. Chen}, the authors characterize 2-local Lie derivations on $B(X)$. Later, L. Chen, and F. Lu \cite{L. Chen2} prove that every local Lie derivation from a nest algebra $alg \mathcal N$ into $B(H)$ is a Lie derivation, where $\mathcal N$ is a nest on the Hilbert space $H$. Quite recently, L. Liu \cite{L. Liu} characterizes 2-local Lie derivations on a semi-finite factor von Neumann algebra with dimension greater than 4.

In Section 3, we study local and 2-local Lie derivations on some algebras by using a new technique. On the algebras including factor von Neumann algebras, finite von Neumann algebras, nest algebras, UHF(uniformly hyperfinite) algebras, and the Jiang-Su algebra, we prove that every local Lie derivation is a Lie derivation. On the algebras including factor von Neumann algebras, UHF algebras, and the Jiang-Su algebra, we prove that every 2-local Lie derivation is a Lie derivation. Besides, for a finite von Neumann algebra $\mathcal{A}$ which is not a factor, we construct an example of 2-local Lie derivation but not a Lie derivation on $\mathcal{A}$.

\section{2-Local derivations}\

Through this paper, we denote by $M_n(\mathcal{A})$ the set of all matrices $(x_{i,j})_{n\times n}$, where $x_{i,j}\in\mathcal{A}$. Clearly, if $\mathcal{A}$ is a Banach algebra, then so is $M_n(\mathcal{A})$, and if $\mathcal{M}$ is an $\mathcal{A}$-bimodule, then $M_n(\mathcal{M})$ is also an $M_n(\mathcal{A})$-bimodule.

\begin{theorem}
Let $\mathcal{A}$ be a unital Banach algebra and $\mathcal{M}$ be a unital $\mathcal{A}$-bimodule.
If each Jordan derivation from $\mathcal{A}$ into $\mathcal{M}$ is an inner derivation, then each 2-local derivation $\Delta$ from $M_n(\mathcal{A})(n>2)$ into $M_n(\mathcal{M})$ is a derivation.
\end{theorem}

Let $\{e_{i,j}\}_{i,j=1}^{n}$ be the system of matrix units in $M_n(\mathcal{A})$ and denote by $span\{e_{i,j}\}_{i,j=1}^{n}$ the linear span of the set $\{e_{i,j}\}_{i,j=1}^{n}$. Besides, we denote the subalgebra of all diagonal matrices in $M_n(\mathcal{A})$ by $D_n(\mathcal{A})$.

To prove Theorem 2.1, we need several lemmas. Firstly, we give the following two lemmas. Since the proofs are completely similar as the proofs of Lemma 2.2 and Lemma 2.9 in \cite{S. Ayupov 5},
we omit it.

\begin{lemma}
Let $\mathcal{A}$ be a unital Banach algebra and $\mathcal{M}$ be a unital $\mathcal{A}$-bimodule.
If each derivation from $\mathcal{A}$ into $\mathcal{M}$ is an inner derivation, then each derivation from $M_n(\mathcal{A})$ into $M_n(\mathcal{M})$ is also an inner derivation.
\end{lemma}

\begin{lemma}
Let $\mathcal{A}$ be a unital Banach algebra and $\mathcal{M}$ be a unital $\mathcal{A}$-bimodule.
If each Jordan derivation from $\mathcal{A}$ into $\mathcal{M}$ is an inner derivation,
 then for each 2-local derivation $\Delta$ from $M_n(\mathcal{A})(n>2)$ into $M_n(\mathcal{M})$, there exists an element $a\in M_n(\mathcal{M})$ such that $\Delta|_{D_n(\mathcal{A})}=D_a|_{D_n(\mathcal{A})}$ and $\Delta|_{span\{e_{i,j}\}_{i,j=1}^{n}}=D_a|_{span\{e_{i,j}\}_{i,j=1}^{n}}$.
\end{lemma}

In following Lemmas 2.4-2.7, the conditions of Theorem 2.1 hold. We consider a 2-local derivation $\delta$ from $M_n(\mathcal{A})$ into $M_n(\mathcal{M})$ such that $\delta|_{D_n(\mathcal{A})}=0$ and $\delta|_{span\{e_{i,j}\}_{i,j=1}^{n}}=0$.

In addition, from the definition, it is easy to see that $\delta$ is homogeneous, i.e. $\delta(\lambda x)=\lambda\delta(x)$, for each $x\in M_n(\mathcal{A})$ and $\lambda\in\mathbb{C}$. Thus, if necessary, we can assume that $\|x_{i,j}\|<1$, and so $e+x_{i,j}$ is invertible in $\mathcal{A}$, where $1\leq i,j\leq n$, $x_{i,j}$ is the $(i,j)-$ entry of $x$ and $e$ is the unit of $\mathcal{A}$.

For each $x=(x_{i,j})_{n\times n}\in M_n(\mathcal{A})$ , we denote by $\widehat{x_{i,j}}$ the matrix in $M_n(\mathcal{A})$ such that the $(i,j)-$ entry is $x_{i,j}$ and the others are zero, i.e.
\[\widehat{x_{i,j}}=e_{i,i}xe_{j,j}=\begin{pmatrix}
0 & \dots & 0 & \dots & 0 \\
\vdots & ~ & \vdots & ~ & \vdots \\
0 & \dots & x_{i,j} & \dots & 0 \\
\vdots & ~ & \vdots & ~ & \vdots \\
0 & \dots & 0 & \dots & 0 \\
\end{pmatrix}\]

\begin{lemma}
For all $x\in M_n(\mathcal{A})$, $[\delta(x)]_{i,i}=0$, $i=1,2,\cdot\cdot\cdot,n$.
\end{lemma}

\begin{proof}
Let $y=e_{i,i}+\widehat{x_{i,i}}$. By lemma 2.2, there exists an element $a$ in $M_n(\mathcal{M})$ such that $\delta(x)=[a,x],\delta(y)=[a,y]$. Since $y\in D_n(\mathcal{A})$, we have $\delta(y)=0$.
For any $k\neq i$,
$$0=[\delta(y)]_{i,k}=[a,y]_{i,k}=\sum_{j=1}^{n}(a_{i,j}y_{j,k}-y_{i,j}a_{j,k})=-y_{i,i}a_{i,k}=-(e+x_{i,i})a_{i,k}.$$
Since $e+x_{i,i}$ is invertible, it follows that $a_{i,k}=0$.
Similarly, $a_{k,i}=0$, for any $k\neq i$. Now we can obtain
$$[\delta(x)]_{i,i}=[a,x]_{i,i}=\sum_{k=1}^{n}(a_{i,k}x_{k,i}-x_{i,k}a_{k,i})=[a_{i,i},x_{i,i}]=[a,y]_{i,i}=0.$$
The proof is complete.
\end{proof}

\begin{lemma}
For each $x\in M_n(\mathcal{A})$, let
\[ y=\begin{pmatrix}
0 & \dots & x_{1,j} & \dots & 0 \\
\vdots & ~ & \vdots & ~ & \vdots \\
x_{i,1} & \dots & x_{i,j} & \dots & x_{i,n} \\
\vdots & ~ & \vdots & ~ & \vdots \\
0 & \dots & x_{n,j} & \dots & 0 \\
\end{pmatrix}\]
i.e. $y=\sum_{k=1}^{n}(\widehat{x_{i,k}}+\widehat{x_{k,j}})-\widehat{x_{i,j}}.$
Then $[\delta(x)]_{i,j}=[\delta(y)]_{i,j}$.
\end{lemma}

\begin{proof}
Take an element $a$ from $M_n(\mathcal{M})$ such that $\delta(x)=[a,x]$, and $\delta(y)=[a,y]$. It is easy to verify the above result by direct calculation.
\end{proof}

\begin{lemma}
For all $x\in M_n(\mathcal{A})$, $[\delta(x)]_{i,j}=[\delta(\widehat{x_{i,j}})]_{i,j}$, whenever $i\neq j$.
\end{lemma}

\begin{proof}
Let $$y=\sum_{k=1}^{n}(\widehat{x_{i,k}}+\widehat{x_{k,j}})-\widehat{x_{i,j}}.$$ By Lemma 2.5, $[\delta(x)]_{i,j}=[\delta(y)]_{i,j}$. Let $s\neq i$, and $z=y-\widehat{x_{s,j}}+e_{s,i}$, i.e.
\[ z=\begin{pmatrix}
0 & \dots & 0 & \dots & x_{1,j} & \dots & 0 \\
\vdots & ~ & \vdots & ~ & \vdots & ~ & \vdots\\
x_{i,1} & \dots & x_{i,i} & \dots & x_{i,j} & \dots & x_{i,n} \\
\vdots & ~ & \vdots & ~ & \vdots & ~ & \vdots\\
0 & \dots & e & \dots & 0 & \dots & 0\\
\vdots & ~ & \vdots & ~ & \vdots & ~ & \vdots\\
0 & \dots & 0 & \dots & x_{n,j} & \dots & 0 \\
\end{pmatrix}\]
Take an element $a$ in $M_n(\mathcal{M})$ such that $\delta(y)=[a,y],\delta(z)=[a,z]$. Then
$$[\delta(y)]_{i,i}-[\delta(z)]_{i,i}=[a,y-z]_{i,i}=[a,\widehat{x_{s,j}}-e_{s,i}]_{i,i}=-a_{i,s}.$$
By Lemma 2.4, $[\delta(y)]_{i,i}=[\delta(z)]_{i,i}=0$, and so $a_{i,s}=0$. Thus
$$[\delta(y)]_{i,j}-[\delta(z)]_{i,j}=[a,y-z]_{i,j}=[a,\widehat{x_{s,j}}-e_{s,i}]_{i,j}=a_{i,s}x_{s,j}=0,$$
i.e. $[\delta(y)]_{i,j}=[\delta(z)]_{i,j}$. Let
\[ w=z-e_{s,i}=\begin{pmatrix}
0 & \dots & 0 & \dots & x_{1,j} & \dots & 0 \\
\vdots & ~ & \vdots & ~ & \vdots & ~ & \vdots\\
x_{i,1} & \dots & x_{i,i} & \dots & x_{i,j} & \dots & x_{i,n} \\
\vdots & ~ & \vdots & ~ & \vdots & ~ & \vdots\\
0 & \dots & 0 & \dots & 0 & \dots & 0\\
\vdots & ~ & \vdots & ~ & \vdots & ~ & \vdots\\
0 & \dots & 0 & \dots & x_{n,j} & \dots & 0 \\
\end{pmatrix}\]
again by Lemma 2.5, we have $[\delta(z)]_{i,j}=[\delta(w)]_{i,j}$. Now we have $[\delta(y)]_{i,j}=[\delta(y-\widehat{x_{s,j}})]_{i,j}$. Repeating the above steps, we can obtain that $[\delta(y)]_{i,j}=[\delta(u)]_{i,j}$, where
\[ u=\sum_{k=1}^{n}\widehat{x_{i,k}}=\begin{pmatrix}
0 & \dots & 0 & \dots & 0 \\
\vdots & ~ & \vdots & ~ & \vdots \\
x_{i,1} & \dots & x_{i,j} & \dots & x_{i,n} \\
\vdots & ~ & \vdots & ~ & \vdots \\
0 & \dots & 0 & \dots & 0 \\
\end{pmatrix}\]
Similarly, we can show that
$$[\delta(u)]_{i,j}=[\delta(u-\widehat{x_{i,s}}+e_{j,s})]_{i,j}=[\delta(u-\widehat{x_{i,s}})]_{i,j},for s\neq j.$$
It follows that $[\delta(u)]_{i,j}=[\delta(\widehat{x_{i,j}})]_{i,j}$ by taking other indexes successively. Hence
$$[\delta(x)]_{i,j}=[\delta(y)]_{i,j}=[\delta(u)]_{i,j}=[\delta(\widehat{x_{i,j}})]_{i,j}.$$
The proof is complete.
\end{proof}

\begin{lemma}
For all $x\in M_n(\mathcal{A})$, $[\delta(x)]_{i,j}=0$, whenever $i\neq j$.
\end{lemma}

\begin{proof}
Assume that $i<j$. Let
\[ v=\sum_{k=1}^{n-1}e_{k+1,k}=\begin{pmatrix}
0 & ~ & ~ & ~ & ~ \\
e & 0 & ~ & ~ & ~ \\
0 & e & 0 & ~ & ~\\
\vdots & ~ & \ddots & \ddots & ~  \\
0 & ~ & \dots & e & 0 \\
\end{pmatrix}\]
It is not difficult to check that every matrix $a$ in $M_n(\mathcal{M})$ commuting with $v$ satisfies the following properties:
$$a_{k,k}=a_{s,s},~~a_{k,s}=0,whenever~~k<s.$$
Let $y$ be a matrix in $M_n(\mathcal{A})$ such that $y_{i,i}=e+x_{i,j},y_{i,j}=x_{i,j}$ and the other entries of $y$ are all zero. Take an element $a$ from $M_n(\mathcal{M})$ such that $\delta(y)=[a,y],\delta(v)=[a,v]$. Since $v\in span\{e_{i,j}\}_{i,j=1}^{n}$, $\delta(v)=[a,v]=0$. Thus $a_{i,i}=a_{j,j}$, and $a_{i,j}=0$. Since $y_{j,i}=0$, by Lemma 2.6, $[\delta(y)]_{j,i}=0$. Then $$0=[\delta(y)]_{j,i}=[a,y]_{j,i}=a_{j,i}y_{i,i}=a_{j,i}(e+x_{i,j}).$$
Since $e+x_{i,j}$ is invertible, it implies that $a_{j,i}=0$. Thus
$$[\delta(y)]_{i,i}=[a,y]_{i,i}=a_{i,i}y_{i,i}-y_{i,i}a_{i,i}-y_{i,j}a_{j,i}=a_{i,i}x_{i,j}-x_{i,j}a_{i,i}.$$
Therefore
$$[\delta(y)]_{i,j}=[a,y]_{i,j}=a_{i,i}y_{i,j}-y_{i,i}a_{i,j}-y_{i,j}a_{j,j}=a_{i,i}x_{i,j}-x_{i,j}a_{i,i}=[\delta(y)]_{i,i}=0.$$
Again by Lemma 2.6,
$$[\delta(x)]_{i,j}=[\delta(y)]_{i,j}=0.$$
Similarly we can show that $[\delta(x)]_{i,j}=0$ if $i>j$.
The proof is complete.
\end{proof}
Now we are in position to prove Theorem 2.1.
\begin{proof}[\bf Proof of Theorem 2.1]
According to Lemma 2.3, there exists a derivation $D_a$ such that $\Delta|_{D_n(\mathcal{A})}=D_a|_{D_n(\mathcal{A})}$ and $\Delta|_{span\{e_{i,j}\}_{i,j=1}^{n}}=D_a|_{span\{e_{i,j}\}_{i,j=1}^{n}}$. Let $\delta=\Delta-D_a$. Then by Lemmas 2.4 and 2.7, we have $\delta=0$, i.e. $\Delta=D_a$. It follows that $\Delta$ is a derivation.
\end{proof}

The following corollaries are some specific examples for applying Theoerm 2.1.

\begin{corollary}
Let $\mathcal{A}$ be a unital commutative C*-algebra and $\mathcal{M}$ be a unital $\mathcal{A}$-bimodule.
Then each 2-local derivation from $M_n(\mathcal{A})(n>2)$ into $M_n(\mathcal{M^{*}})$ is a derivation.
\end{corollary}
\begin{proof}
Since $\mathcal{M}$ is a unital $\mathcal{A}$-bimodule, so is $\mathcal{M^{*}}$. We know $\mathcal{A}$ is amenable, that is each derivation from $\mathcal{A}$ into $\mathcal{M^{*}}$ is an inner derivation \cite{V. Runde}. Besides, there is a classic result that each Jordan derivation from a C*-algebra into its bimodule is a derivation (see \cite{N. Jacobson}). By Theorem 2.1, the proof is complete.
\end{proof}

\begin{corollary}
Let $\mathcal{A}$ be a C*-algebra. Then each 2-local derivation from $M_n(\mathcal{A})(n>2)$ into $M_n(\mathcal{A^{*}})$ is a derivation.
\end{corollary}
\begin{proof}
Every C*-algebra $\mathcal{A}$ is weakly amenable, that is, each derivation from $\mathcal{A}$ into $\mathcal{A^{*}}$ is an inner derivation \cite{V. Runde}. By Theorem 2.1, the proof is complete.
\end{proof}

\begin{corollary}
Each 2-local derivation on the matrix algebra $M_n(\mathcal{A})(n>2)$  is a derivation, if the algebra $\mathcal{A}$ satisfies one of the following conditions:\\
$(1)$ $\mathcal{A}=alg\mathcal{L}$, where $\mathcal{L}$ is a subspace lattice on a Hilbert space $\mathcal{H}$ with the property that $0_{+} \neq 0$ or $\mathcal{H}_{-} \neq \mathcal{H}$;\\
$(2)$ $\mathcal{A}=alg(N_1\bigotimes N_2\bigotimes \cdots \bigotimes N_n)$, where each $N_i$ is a nest, $i=1,2,\cdot\cdot\cdot,n$.
\end{corollary}

\begin{proof}
(1) According to \cite[Theorem 2.1]{F. Lu1}, each Jordan derivation on $\mathcal{A}$ is a derivation. And see in \cite{J. Li} , each derivation on $\mathcal{A}$ is an inner derivation.
(2) It is known that $N_1\bigotimes N_2\bigotimes \cdots \bigotimes N_n$ is still a CSL, so Jordan derivations on $\mathcal{A}$ are derivations \cite{F. Lu2}. In \cite{F. Gilfeather}, F. Gilfeather, A. Hopenwasser and D. Larson prove that the cohomology vanishes for CSL algebras whose lattices are generated by finite independent nests. Applying this result, we immediately obtain that derivations on $\mathcal{A}$ are inner derivations. By Theorem 2.1, the proof is complete.
\end{proof}

\begin{corollary}
Let $\mathcal{A}$ be a nest algebra of infinite multiplicity. Then each 2-local derivation on $\mathcal{A}$ is a derivation.
\end{corollary}

\begin{proof}
Since $\mathcal{A}$ is of infinite multiplicity, it is isomorphic to $M_n(\mathcal{B})$ for some nest algebra $\mathcal{B}$ and some integer $n>2$. It is known that
each derivation on nest algebras is an inner derivation \cite{E. Christensen}. By Theorem 2.1, the proof is complete.
\end{proof}

\begin{lemma}
Let $\mathcal{A}=\bigoplus_{i=1}^{\infty}\mathcal{A}_i$ be a Banach algebra with the inner derivation property, i.e. all derivations on $\mathcal{A}$ are inner derivations. If each 2-local derivation on $\mathcal{A}_i$
is a derivation for any $i\in \mathbb{N}$, then each 2-local derivation on $\mathcal{A}$ is also a derivation.
\end{lemma}

\begin{proof}
We only consider the case $\mathcal A=\mathcal{A}_1\bigoplus \mathcal{A}_2$. For a 2-local derivation $\delta$ on $\mathcal{A}$, denote
by $\delta_i$ the restriction of $\delta$ in $\mathcal{A}_i$, $i=1,2$. One can easily verify that $\delta(a_i)=\delta_i(a_i)\in\mathcal{A}_i$ for each
$a_i\in\mathcal{A}_i$, moreover, $\delta_i$ is a 2-local derivation on $\mathcal{A}_i$,
and thus a derivation.

For each $a\in \mathcal A$, we write $a=a_1+a_2$, where $a_i\in\mathcal{A}_i$, $i=1,2$. By the definition of 2-local derivation, there exists an element
$x\in\mathcal A$ such that $\delta(a)=[a,x]=[a_1,x_1]+[a_2,x_2]$ and $\delta(a_1)=[a_1,x]=[a_1,x_1]$. Similarly, there exists an element
$y\in\mathcal A$ such that $\delta(a)=[a,y]=[a_1,y_1]+[a_2,y_2]$ and $\delta(a_2)=[a_2,y]=[a_2,y_2]$.
So we can obtain $$\delta(a)=\delta(a_1)+[a_2,x_2]=[a_1,y_1]+\delta(a_2).$$
It follows that $\delta(a_1)=[a_1,y_1]$, $\delta(a_2)=[a_2,x_2]$ and thus $\delta(a)=\delta(a_1)+\delta(a_2)$. Hence $\delta$ is linear.
Moreover, we can obtain
\begin{align*}
\delta(ab)
&=\delta(a_1b_1+a_2b_2)\\
&=\delta(a_1b_1)+\delta(a_2b_2)\\
&=\delta(a_1)b_1+a_1\delta(b_1)+\delta(a_2)b_2+a_2\delta(b_2)\\
&=(\delta(a_1)+\delta(a_2))(b_1+b_2)+(a_1+a_2)(\delta(b_1)+\delta(b_2))\\
&=\delta(a)b+a\delta(b)
\end{align*}
Hence $\delta$ is a derivation on $\mathcal A$. The proof is complete.
\end{proof}

\begin{corollary}
Let $\mathcal{A}$ be a von Neumann algebra without direct summand of type $I_1$ and $I_2$. Then each 2-local derivation on $\mathcal{A}$ is a derivation.
\end{corollary}

\begin{proof}
It is well known that $\mathcal{A}=\bigoplus_{i=1}^{\infty}\mathcal{A}_i$, where each $\mathcal{A}_i$ is isomorphic to a matrix algebra $M_n(\mathcal{B})$
for some von Neumann algebra $\mathcal B$ and some integer $n>2$. By Theorem 2.1 and Lemma 2.12, the result follows.
\end{proof}

\begin{remark}
In \cite{S. Ayupov 3} and \cite{S. Ayupov 4}, the authors show that each 2-local derivation on a von Neumann algebra is a derivation in other ways. By comparison, our proof is more simple. However, we can not handle the case for type $I_2$. In addition, for the case of type $I_1$, the result is trival, since each derivation on an abelian von Neumann algebra is zero.
\end{remark}

\begin{corollary}
Let $\mathcal{A}$ be a unital algebra with the inner derivation property and $n\geq6$ be a positive integer but not a prime number. Then each 2-local derivation on the matrix algebra $M_n(\mathcal{A})$ is a derivation.
\end{corollary}

\begin{proof}
Suppose $n=rt$, where $r>2$ and $t>1$. Then $M_n(\mathcal{A})$ is isomorphic to $M_r(M_t(\mathcal{A}))$. In \cite{R. Alizadeh}, the author proves that each Jordan derivation on
$M_t(\mathcal{A})(t>1)$ is a derivation. By Theorem 2.1, the proof is complete.
\end{proof}

Through the same technique with the proof of Theorem 2.1, we can show the following theorem.

\begin{theorem}
Let $\mathcal{A}$ be a unital Banach algebra such that:\\
$(1)$ each Jordan derivation on $\mathcal{A}$ is a derivation;\\
$(2)$ for each derivation $D$ on $\mathcal{A}$, there exists an element $a$ in $\mathcal{B}$ such that $D(x)=[a,x]$ for all $x\in \mathcal{A}$, where $\mathcal{B}$ is an algebra containing $\mathcal{A}$.\\
Then each 2-local derivation on $M_n(\mathcal{A})(n>2)$ is a derivation.
\end{theorem}

\begin{corollary}
Let $\mathcal{A}$ be a C*-algebra. Then each 2-local derivation on $M_n(\mathcal{A})(n>2)$ is a derivation.
\end{corollary}

Let $\mathcal{A}$ be a C*-algebra, $\pi$ a representation of $\mathcal{A}$, and $\mathcal{M}$ the
von Neumann algebra generated by $\pi(\mathcal A)$. Then $\pi$ is said to be a \emph{traceable representation} if there exists a faithful normal
trace $\tau$ on $\mathcal{M}^+$ such that $\pi(\mathcal A)\bigcap \mathcal{M}_\tau$ is weakly dense in $\mathcal{M}$, where $\mathcal{M}_\tau$
denotes the span of the set $\{M\in \mathcal{M}^+: \tau(M)<\infty\}$. Especially, a finite representation $\pi$(i.e. the von Neumann algebra
generated by $\pi(\mathcal A)$ is finite) can be regarded as a specific traceable representation.

\begin{theorem}
Let $\mathcal{A}$ be a C*-algebra with a faithful traceable representation $\pi$. Then each 2-local derivation on $\mathcal{A}$ is a derivation.
\end{theorem}

\begin{proof}
Let $\mathcal{M}$ be the von Neumann algebra generated by $\pi(\mathcal A)$, and  $\tau$ a faithful normal
trace on $\mathcal{M}^+$ such that $\mathcal{A}_\tau$ is weakly dense in $\mathcal{M}$, where $\mathcal{A}_\tau$
denotes $\pi(\mathcal A)\bigcap \mathcal{M}_\tau$. It is known that $\mathcal{M}_\tau$ is a two-side ideal of $\mathcal M$,
thus $\mathcal{A}_\tau$ is also a two-side ideal of $\pi(\mathcal A)$.

For a 2-local derivation $\Delta$ on $\mathcal{A}$, define
$\delta=\pi\circ\Delta\circ\pi^{-1}$. Then $\delta$ is a 2-local derivation on $\pi(\mathcal A)$.
For each $x\in \pi(\mathcal A)$ and $y\in \mathcal{A}_\tau$, there exists an element $m\in \mathcal M$
such that $\delta(x)=[m,x]$ and $\delta(y)=[m,y]$. Hence
\begin{align*}
\tau(\delta(x)y)
&=\tau((mx-xm)y)=\tau(mxy)-\tau(xmy)\\
&=\tau(xym)-\tau(xmy)=\tau(x(ym-my))\\
&=-\tau(x\delta(y)).
\end{align*}
For each $a,b\in\pi(\mathcal A)$ and $x\in \mathcal{A}_\tau$, we have
\begin{align*}
\tau(\delta(a+b)x)
&=-\tau((a+b)\delta(x))=-\tau(a\delta(x))-\tau(b\delta(x))\\
&=\tau(\delta(a)x)+\tau(\delta(b)x)\\
&=\tau((\delta(a)+\delta(b))x).
\end{align*}
It means that $\tau(ux)=0$ for any $x\in \mathcal{A}_\tau$, where
$u$ denotes $\delta(a+b)-\delta(a)-\delta(b)$.
For any $y\in \mathcal{A}_\tau$, by taking $x=yy^*u^*$, we obtain
$\tau(uyy^*u^*)=0$. Since $\tau$ is faithful and $\mathcal{A}_\tau$ is weakly dense,
we have $uy=0$ and thus $u=0$.

Hence $\delta$ is additive, and thus a local derivation, moreover, a derivation. It follows that
$\Delta=\pi^{-1}\circ\delta\circ\pi$ is a derivation on $\mathcal A$. The proof is complete.
\end{proof}

As direct applications of the above theorem, we have the following corollaries.

\begin{corollary}
Each 2-local derivation on a UHF C*-algebra is a derivation.
\end{corollary}

\begin{corollary}
Each 2-local derivation on the Jiang-Su algebra is a derivation.
\end{corollary}

\section{Local and 2-local Lie derivations}\

In this section, we study local and 2-local Lie derivations on some algebras. The main results are as follows.
\begin{theorem}
Let $\mathcal{B}$ be a Banach algebra over $\mathbb{C}$ satisfying the following conditions:\\
$(1)$ each Lie derivation on $\mathcal{B}$ is standard;\\
$(2)$ each derivation on $\mathcal{B}$ is inner;\\
$(3)$ each local derivation on $\mathcal{B}$ is a derivation;\\
$(4)$ every nonzero element in $\mathcal{Z}(\mathcal{B})$$(the ~center ~of ~\mathcal{B})$ can not be written in the form $\sum_{i=1}^{n}[A_i,B_i]$, where $A_i,B_i\in \mathcal{B}.$\\
Then each local Lie derivation $\varphi$ on $\mathcal{B}$ is a Lie derivation.
\end{theorem}

\begin{proof}
By the definition of local Lie derivation, for any $A\in \mathcal{B}$, there exists a Lie derivation $\varphi_{A}$ such that $\varphi_{A}(A)=\varphi(A)$. Since $\varphi_{A}$ is standard, $\varphi_{A}=\delta_A+\tau_A$, where $\delta_A$ is a derivation and $\tau_A$ is a linear mapping from $\mathcal{B}$ into $\mathcal{Z}(\mathcal{B})$ such that $\tau_A[X,Y]=0$, for all $X,Y\in \mathcal{B}.$ Thus
$$\varphi(A)=\delta_A(A)+\tau_A(A)=[A,T_A]+\tau_A(A)$$
for some $T_A\in \mathcal B$, which means that $\varphi$ has a decomposition at each point. In the following we prove that the decomposition is unique. Suppose $$\varphi(A)=\delta_A^{'}(A)+\tau_A^{'}(A)=[A,T^{'}_A]+\tau_A^{'}(A).$$ Then
$$[A,T_A]+\tau_A(A)=[A,T^{'}_A]+\tau_A^{'}(A).$$ Thus
$$[A,T_A]-[A,T^{'}_A]=\tau_A^{'}(A)-\tau_A(A).$$
Since every nonzero element in $\mathcal{Z}(\mathcal{B})$ can not be written in the form $\sum_{i=1}^{n}[A_i,B_i]$, we obtain
$$[A,T_A]-[A,T^{'}_A]=\tau_A^{'}(A)-\tau_A(A)=0,$$
and so
$$\delta_A(A)=\delta_A^{'}(A),\tau_A(A)=\tau_A^{'}(A).$$
Now we can define
$$\delta(A)=\delta_A(A)=[A,T_A] ~~\mathrm{and}~~ \tau(A)=\tau_A(A),$$
for every $A\in \mathcal{B}$.
We shall prove $\delta$ and $\tau$ are linear.
For all $A,B\in \mathcal{B}$, since $$\varphi(A+B)=\varphi(A)+\varphi(B),$$
we obtain that
$$[A,T_A]+[B,T_B]-[A+B,T_{A+B}]=\tau(A+B) -\tau(A) -\tau(B) .$$
Hence $$[A+B,T_{A+B}]=[A,T_A]+[B,T_B],~~\tau(A+B) =\tau(A) +\tau(B).$$
It means that $\delta$ and $\tau$ are additive. Similarly, we can prove that $\delta$ and $\tau$ are homogeneous, and so linear. Hence $\delta$ is a local derivation, by assumption, a derivation, and $\tau$ is a linear mapping from $\mathcal{B}$ into $\mathcal{Z}(\mathcal{B})$ such that $\tau[X,Y]=0$, for all $X,Y\in \mathcal{B}.$ It follows that $\varphi$ is a Lie derivation. The proof is complete.
\end{proof}

\begin{corollary}
Each local Lie derivation on a factor von Neumann algebra $\mathcal{M}$ is a Lie derivation.
\end{corollary}

\begin{proof}
It is known that every element in $\mathcal{M}$ can be written in the form $\lambda I+\sum_{i=1}^{n}[A_i,B_i]$, where $\lambda\in \mathbb{C}$ and $A_i,B_i\in \mathcal{M}$(see, \cite{M. Bresar}). If the unit $I$ can not be written in the form $\sum_{i=1}^{n}[A_i,B_i]$, where $A_i,B_i\in \mathcal{M}$, then by Theorem 3.1, the result follows. Otherwise, every element in $\mathcal{M}$ has the form $\sum_{i=1}^{n}[A_i,B_i]$. Hence each Lie derivation is a derivations, so every local Lie derivation is a Lie derivation.
\end{proof}

The following result is proved in \cite{L. Chen2}. We show it by Theorem 3.1 as an alternative method.
\begin{corollary}
Each local Lie derivation on a nest algebra $\mathcal{A}=alg \mathcal{N}$ is a Lie derivation.
\end{corollary}

\begin{proof}
If $\mathcal{A}$ is of infinite multiplicity, then every $A\in \mathcal{A}$ has the form $\sum_{i=1}^{n}[A_i,B_i]$, where $A_i,B_i\in \mathcal{A}$(\cite[Theorem 4.7]{L. Marcoux}). In this case, all Lie derivations are derivations, and so the result is trival.

Otherwise, suppose $\mathcal{A}$ is not of infinite multiplicity. Then $\mathcal{N}$ has a finite-dimensional atom $P$. It is known that $\mathcal{A}$ satisfies the conditions (1) to (3) of Theorem 3.1(see in \cite{F. Lu3, E. Christensen, D. Hadwin}), so it is sufficient to check the condition (4). We know $\mathcal{Z}(\mathcal{A})=\mathbb{C}I$. Assume that the unit $I$ can be written in the form $\sum_{i=1}^{n}[A_i,B_i]$. It is known that the linear map $\sigma:\mathcal{A}\rightarrow P\mathcal{A}P$ defined by $\sigma(A)=PAP$ for all $A \in\mathcal{A}$ is a homomorphism. So we have
$$P=PIP=\sum_{i=1}^{n}P[A_i,B_i]P=\sum_{i=1}^{n}[\sigma(A_i),\sigma(B_i)].$$
It follows that there is no trace on $P\mathcal{A}P$, which contradicts that $P\mathcal{A}P$ has finite dimension. Therefore $I$ can not be written in the form $\sum_{i=1}^{n}[A_i,B_i]$.
The proof is complete.
\end{proof}

\begin{theorem}
Let $\mathcal{B}$ be an algebra over $\mathbb{C}$ satisfying the following conditions:\\
$(1)$ each Lie derivation on $\mathcal{B}$ is standard;\\
$(2)$ each derivation on $\mathcal{B}$ is inner;\\
$(3)$ each local derivation on $\mathcal{B}$ is a derivation;\\
$(4)$ there exists a center-valued trace on $\mathcal{B}$.\\
Then each local Lie derivation $\varphi$ on $\mathcal{B}$ is a Lie derivation.
\end{theorem}
\begin{proof}
Since there exists a center-valued trace on $\mathcal{B}$, every nonzero element in $\mathcal{Z}(\mathcal{B})$
 can not be written in the form $\sum_{i=1}^{n}[A_i,B_i]$, where $A_i,B_i\in \mathcal{B}.$ By Theorem 3.1, the result follows.
\end{proof}

\begin{corollary}
Each local Lie derivation on a finite von Neumann algebra $\mathcal{M}$ is a Lie derivation.
\end{corollary}

\begin{proof}
Since $\mathcal{M}$ is finite, there exists a central valued trace on $\mathcal{M}$. By Theorem 3.4, the result follows.
\end{proof}

\begin{remark}
For a properly infinite von Neumann algebra $\mathcal{M}$, since every element in $\mathcal{M}$ has the form $\sum_{i=1}^{n}[A_i,B_i]$, where $A_i,B_i\in \mathcal{M}$(see, \cite{M. Bresar}), all Lie derivations are derivations, so every local Lie derivation is a Lie derivation.
\end{remark}

Also as applications of Theorem 3.4, we can immediately obtain the following two corollaries.
\begin{corollary}
Each local Lie derivation on a UHF C*-algebra $\mathcal{A}$ is a Lie derivation.
\end{corollary}

\begin{corollary}
Each local Lie derivation on the Jiang-Su algebra $\mathcal{Z}$ is a Lie derivation.
\end{corollary}

\begin{theorem}
Let $\mathcal{B}$ be an algebra over $\mathbb{C}$ satisfying the following conditions:\\
$(1)$ each Lie derivation on $\mathcal{B}$ is standard;\\
$(2)$ each derivation on $\mathcal{B}$ is inner;\\
$(3)$ each 2-local derivation on $\mathcal{B}$ is a derivation;\\
$(4)$ $\mathcal{Z}(\mathcal{B})=\mathbb{C}I$, where $\mathcal{Z}(\mathcal{B})$ is the center of $\mathcal{B}$;\\
$(5)$ every element in $\mathcal{B}$ can be written in the form $\lambda I+\sum_{i=1}^{n}[A_i,B_i]$, where $\lambda\in \mathbb{C}$ and $A_i,B_i\in \mathcal{B}.$\\
Then each 2-local Lie derivation $\varphi$ on $\mathcal{B}$ is a Lie derivation.
\end{theorem}

\begin{proof}
Similarly with the foregoing discussion, if $I$ can be written in the form $\sum_{i=1}^{n}[A_i,B_i]$, then the result is trival. We assume that $I$ can not be written in the form $\sum_{i=1}^{n}[A_i,B_i]$. In this case, we have proved that $\varphi$ has a unique decomposition at each point $A$ in $\mathcal{B}$, i.e.
$$\varphi(A)=\varphi_A(A)=\delta_A(A)+\tau_A(A),$$
where $\varphi_{A}$ is a Lie derivation, $\delta_A$ is a derivation and $\tau_A$ is a linear mapping from $\mathcal{B}$ into $\mathcal{Z}(\mathcal{B})$ such that $\tau_A[X,Y]=0$, for all $X,Y\in \mathcal{B}.$
Similarly we can define
$$\delta(A)=\delta_A(A)~~\mathrm{and}~~ \tau(A)=\tau_A(A),$$
for all $A\in \mathcal{B}$.
Now it is sufficient to show that $\delta$ is a derivation and $\tau$ is linear.
Given $A$ and $B$ in $\mathcal{B}$, there exists a Lie derivation $\varphi_{A,B}$ such that
$$\varphi(A)=\varphi_{A,B}(A)=\delta_{A,B}(A)+\tau_{A,B}(A),$$ and
$$\varphi(B)=\varphi_{A,B}(B)=\delta_{A,B}(B)+\tau_{A,B}(B),$$ where $\delta_{A,B}$ + $\tau_{A,B}$ is the standard decomposition of $\varphi_{A,B}$.
By the uniqueness of the decomposition, $\delta(A)=\delta_{A,B}(A)$ and $\delta(B)=\delta_{A,B}(B)$. Hence $\delta$ is a 2-local derivation and thus a derivation.
Now Our task is to show that $\tau$ is linear. It is easy to see that $\varphi$ is homogeneous, and so is $\tau$. If both $A$ and $B$ have the form $\sum_{i=1}^{n}[A_i,B_i]$, then the equation
$\tau(A+B)=\tau(A)+\tau(B)$ is trival. Otherwise, suppose $A=\lambda I+\sum_{i=1}^{n}[A_i,B_i],$ where $\lambda\neq0$. Take a Lie derivation $\varphi_{A,A+B}$ such that
$$\varphi(A)=\varphi_{A,A+B}(A)=\delta_{A,A+B}(A)+\tau_{A,A+B}(A),$$ and
$$\varphi(A+B)=\varphi_{A,A+B}(A+B)=\delta_{A,A+B}(A+B)+\tau_{A,A+B}(A+B),$$ where $\delta_{A,A+B}$ + $\tau_{A,A+B}$ is the standard decomposition of $\varphi_{A,A+B}$.
 From $$\tau(A)=\tau_{A,B}(A)=\tau_{A,A+B}(A),$$
we have $\tau_{A,B}(I)=\tau_{A,A+B}(I)$, and thus
$$\tau_{A,B}(B)=\tau_{A,A+B}(B).$$
Now we can obtain that
\begin{align*}
\tau(A+B)&=\tau_{A,A+B}(A+B)=\tau_{A,A+B}(A)+\tau_{A,A+B}(B)\\
&=\tau(A)+\tau_{A,B}(B)=\tau(A)+\tau(B).
\end{align*}
Hence $\tau$ is additive, and so linear. The proof is complete.
\end{proof}

Directly applying the above theorem, we can obtain the following corollaries.

\begin{corollary}
Each 2-local Lie derivation on the Jiang-Su algebra is a Lie derivation.
\end{corollary}

\begin{corollary}
Each 2-local Lie derivation on a UHF C*-algebra is a Lie derivation.
\end{corollary}

\begin{corollary}
Each 2-local Lie derivation on a factor von Neumann algebra is a Lie derivation.
\end{corollary}

\begin{example}
If $\mathcal{M}$ is a finite von Neumann algebra but not a factor, then there exists a 2-local Lie derivation on $\mathcal{M}$ but not a Lie derivation.
\end{example}
\begin{proof}
There exist a center-valued trace $\tau$ on $\mathcal{M}$ and a non-trival central projection $P$ in $\mathcal{M}$.
Suppose the Hilbert space that $\mathcal{M}$ acts on is $H$. Take two unital vectors $x$ and $y$ in $PH$ and $(I-P)H$, respectively. For each element
$M\in \mathcal M$, denote $<\tau(M)x,x>$ and $<\tau(M)y,y>$ by $x_M$ and $y_M$, respectively.

Define a functional $f$ on $\mathcal{M}$ by follows:
$f(M)$ equals to $\frac{x_M^2}{y_M}$ if $y_M\neq0$, else $0$, and define a mapping $\delta$ from $\mathcal{M}$ into $\mathcal{Z}(\mathcal{M})$ by $\delta(M)=f(M)P$.
Obviously, $\delta$ is not linear and thus not a Lie derivation. Now we only need to check that $\delta$ is a 2-local Lie derivation, i.e. for each $A,B$ in $\mathcal{M}$,
there exists a Lie derivation $\phi$ on $\mathcal{M}$ such that $\phi(A)=\delta(A)$ and $\phi(B)=\delta(B)$.

\textbf{case1} $x_Ay_B\neq x_By_A$.

Take $$\phi(M)=\frac{(f(A)y_B-f(B)y_A)x_M+(f(B)x_A-f(A)x_B)y_M}{x_Ay_B-x_By_A}P.$$
It is not difficult to check that $\phi$ is a linear mapping from $\mathcal{M}$ into $\mathcal{Z}(\mathcal{M})$ and $\phi([X,Y])=0$ for all $X,Y$ in $\mathcal{M}$. Hence
$\phi$ is a Lie derivation. Moreover, $\phi(A)=\delta(A)$ and $\phi(B)=\delta(B)$.

\textbf{case2} $x_Ay_B= x_By_A$.

If both $y_A$ and $y_B$ are zero, then it is sufficient to take $\phi\equiv0$. Otherwise, without loss of generality, we may assume $y_A\neq0$, then take $\phi(M)=\frac{x_Ax_M}{y_A}P$.
One can easily check that $\phi$ is a Lie derivation and $\phi(A)=\delta(A)$. If $y_B=0$, then $x_B=\frac{x_Ay_B}{y_A}=0$ and $\phi(B)=0=\delta(B)$.
If $y_B\neq0$, then $\phi(B)=\frac{x_Ax_B}{y_A}P=\frac{x_B^2}{y_B}P=\delta(B)$.

The proof is complete.
\end{proof}

\emph{Acknowledgements}. This paper was partially supported by National Natural Science Foundation of China(Grant No. 11371136).

\end{document}